\documentclass[12pt,a4paper]{amsart}
\usepackage[bottom]{footmisc}
\usepackage{mathrsfs}
\usepackage{amsmath}
\usepackage{amsthm}
\usepackage{amsfonts}
\usepackage{latexsym}
\usepackage{graphicx}
\usepackage{hyperref}
\usepackage{amssymb}
\usepackage{subfigure}
\usepackage{epsf}
\usepackage{float}
\usepackage{hyperref}
\usepackage{fancyhdr}
\allowdisplaybreaks \makeatletter
\def\rightharpoonfill@{\arrowfill@\relbar\relbar\rightharpoonup}
\DeclareRobustCommand{\overrightharpoon}{\mathpalette{\underarrow@\rightharpoonfill@}}
\makeatother

\allowdisplaybreaks

\begin{document}
\newcommand{\beq}{\begin{equation}}
\newcommand{\eneq}{\end{equation}}
\newtheorem{thm}{Theorem}[section]
\newtheorem{coro}[thm]{Corollary}
\newtheorem{lem}[thm]{Lemma}
\newtheorem{prop}[thm]{Proposition}
\newtheorem{defi}[thm]{Definition}
\newtheorem{rem}[thm]{Remark}
\newtheorem{cl}[thm]{Claim}
\title{Analytic solutions for the approximation of $p$-Laplacian problem}
\author{Xiaojun Lu$^{1}$\ \ \ \ Xiaofen Lv$^2$}
\pagestyle{fancy}                   % ÉèÖÃҳü
\lhead{X. Lu and X. Lv}
\rhead{Analytic solution for $p$-Laplacian problem} %\rhead{\small\leftmark}
\thanks{Corresponding author: Xiaojun Lu, Department of Mathematics \& Jiangsu Key Laboratory of
Engineering Mechanics, Southeast University, 210096, Nanjing, China}
\thanks{Email addresses:  lvxiaojun1119@hotmail.de(Xiaojun Lu), lvxiaofen0101@hotmail.com(Xiaofen Lv)}
\date{}
\maketitle
\begin{center}
1. Department of Mathematics \& Jiangsu Key Laboratory of
Engineering Mechanics, Southeast University, 210096, Nanjing, China\\
2. Jiangsu Testing Center for Quality of Construction Engineering
Co., Ltd, 210028, Nanjing, China\\
%3. Faculty of Science and Technology, Federation University
%Australia, Ballarat, VIC 3350, Australia
\end{center}
\begin{abstract}
This paper mainly investigates the analytic solutions for the
approximation of $p$-Laplacian problem. Through an approximation
mechanism, we convert the nonlinear partial differential equation
with Dirichlet boundary into a sequence of minimization problems.
And a sequence of analytic minimizers can be obtained by applying
the canonical duality theory. Moreover, the nonlinear canonical
transformation gives a sequence of perfect dual
maximization(minimization) problems, and further discussion shows
the global extrema for both primal and dual problems.
\end{abstract}
\section{Introduction}
Fractional order operators are very important mathematical models
describing plenty of anomalous dynamic behaviors in applied sciences
\cite{AH,LIONS1,LU}. Let $\Omega$ be a bounded domain in
$\mathbb{R}^n$ with a sufficiently smooth boundary $\Gamma$. In this
manuscript, we are interested in exploring the analytic solutions
for the approximation of the following fractional order
$p$-Laplacian problem(also called $p$-harmonic problem) with
Dirichlet boundary in higher dimensions,
\beq\left\{\begin{array}{cll}
\Delta_p u:={\rm div}(|\nabla u|^{p-2}\nabla u)+f=0 &\ \text{in}&\ \Omega,\\
u=0&\ \text{on}&\ \Gamma,
\end{array}\right.\eneq
where $f\in L^{p^*}(\Omega)$, $1/p+1/p^*=1$. This nonhomogeneous
problem is intensively studied in many multidisciplinary fields,
such as mean curvature analysis($p=1$), compressible fluid in a
homogeneous isotropic rigid porous medium($p={3}/{2}$),
glaciology($p\in(1,{4}/{3}]$), nonlinear elasticity($p>2$) and
probabilistic games($p=\infty$), etc. Interested readers can refer
to \cite{A,SB,MD,LIONS1,LU} to become familiar with more useful
applications in these respects.\\

Well-posedness and numerical simulations through finite element
methods(FEMs) for $p\in(1,\infty)$ are well established
\cite{WB,AM,GL,U}. General $p$-supersolution in the viscosity sense
for the $p$-Laplacian can be found in \cite{JM}. In particular, if
$\Omega$ is a bounded domain of class $C^{1,\beta}$ for some
$\beta\in(0,1)$ and $f\in L^\infty(\Omega)$, the unique weak
solution $u$ of (1) belongs to $C^1_0(\overline{\Omega})$ \cite{CA}.
In effect, from the minimization problem corresponding to (1)\beq
(\mathscr{P}): \displaystyle\min_{u\in
W^{1,p}_0(\Omega)}\Big\{I[u]:=1/p\int_\Omega |\nabla
u|^pdx-\int_\Omega fudx\Big\},\eneq
%where the feasible
%function space $\mathscr{N}$ is defined as
%\[\mathscr{N}:=\Big\{u\in W_0^{1,p}(\Omega):\ u\ \text{is radially
%symmetric}\Big\},\]
it is evident that, for $x\in\overline{\Omega}$, if ${\rm
Sgn}(f)=1$, then there exists a unique positive minimizer; while if
${\rm Sgn}(f)=-1$, then there exists a unique negative minimizer
\cite{SS}. For both Damascelli-Pacella's weak comparison principle
and Cuesta-Taka\v{c}'s strong comparison principle of positive
solutions, please refer to \cite{CA}.\\

Here, we mainly address the nonlinear case
$p\in(1,2)\cup(2,\infty)$. It is worth noticing that, at the
critical points($\nabla u=0$), the operator is degenerate elliptic
for $p>2$ and singular for $p<2$. To tackle the singularity, one
applies the perturbation method(or penalty function method
frequently used in the image processing computation) proposed in
\cite{Evans1,JM}, namely, \beq\left\{\begin{array}{cll}
\Delta_{p,\epsilon} u_\epsilon:={\rm div}(|\nabla u_\epsilon|^2+\chi(p)\epsilon^2)^{(p-2)/2}\nabla u_\epsilon)+f=0 &\ \text{in}&\ \Omega,\\
u_\epsilon=0&\ \text{on}&\ \Gamma,
\end{array}\right.\eneq
where the cut-off function $\chi$ is defined as
\[\chi(p):=\left\{\begin{array}{ll}
1,&p\in(1,2);\\
0,&p\in(2,+\infty).
\end{array}\right.\]
The term $(|\nabla u_\epsilon|^2+\chi(p)\epsilon^2)^{(p-2)/2}$ is
called the {\it transport density}. Clearly, (3) is a highly
nonlinear PDE which is difficult to solve by the direct approach
\cite{JH,LIONS}. As a matter of fact, given the distributed source
term $f\in L^{p^*}(\Omega)$, (3) is the Euler-Lagrange equation of
the minimization problem \beq (\mathscr{P}^{(\epsilon)}):
\displaystyle\min_{u_\epsilon\in
\mathscr{N}}\Big\{I^{(\epsilon)}[u_\epsilon]:=\int_\Omega
L_{\epsilon,p}(\nabla u_\epsilon,u_\epsilon,x)dx=\int_\Omega
H_{\epsilon,p}(\nabla u_\epsilon)dx-\int_\Omega fu_\epsilon
dx\Big\},\eneq where $\mathscr{N}=W^{1,p}_0(\Omega)$, the
function(so-called stored strain energy density) $H_{\epsilon,p}:
\mathbb{R}^n\to\mathbb{R}$ is given by
\[\displaystyle
H_{\epsilon,p}(\overrightarrow{\gamma}):=(|\overrightarrow{\gamma}|^2+\chi(p)\epsilon^2)^{p/2}/p.\]
Moreover,
$L_{\epsilon,p}(\overrightarrow{\gamma},z,x):\mathbb{R}^n\times\mathbb{R}\times
\Omega\to \mathbb{R}$ satisfies the following coercivity inequality
and is convex in the variable $\overrightarrow{\gamma}$, \beq
L_{\epsilon,p}(\overrightarrow{\gamma},z,x)\geq
p_{\epsilon}|\overrightarrow{\gamma}|^p-q_\epsilon,\
\overrightarrow{\gamma}\in\mathbb{R}^n, z\in\mathbb{R}, x\in \Omega,
\eneq for certain constants $p_\epsilon$ and $q_\epsilon$.
$I^{(\epsilon)}$ is called the {\it potential energy functional} and
is sequentially(weakly) lower semicontinuous with respect to the
weak topology of $W^{1,p}_0(\Omega)$. By Rellich-Kondrachov
compactness theorem and Riesz's mean convergence theorem, we have
the following a priori estimates,
\begin{thm}
Assume that $f\in L^{p^*}(\Omega)$ does not change its sign on
$\overline{\Omega}$, that is to say, ${\rm Sgn}(f)=1$ or ${\rm
Sgn}(f)=-1$. Then, there exists at least one minimizer
$\bar{u}_\epsilon\in W^{1,p}_0(\Omega)$ of $(4)$ satisfying \beq
\bar{u}_\epsilon\geq0\ \text{\rm a.e. in}\ \Omega\ \text{\rm for}\
{\rm Sgn}(f)=1,\eneq \beq \bar{u}_\epsilon\leq0\ \text{\rm a.e. in}\
\Omega\ \text{\rm for}\ {\rm Sgn}(f)=-1,\eneq \beq
\bar{u}_\epsilon\to \bar{u}\ \text{in}\ W^{1,p}_0(\Omega)\
\text{as}\ \epsilon\to0, \eneq where $\bar{u}$ is the unique
positive solution to (2).
\end{thm}
Next, we are going to explore the analytic solution of (3) rather
than finite element approximation, which is of great use in the
discussion of Mong-Kantorovich mass transfer problems \cite{Evans1},
etc. Here, let $\Omega=\mathbb{B}(O,R_2)\backslash
\mathbb{B}(O,R_1)$, $R_2>R_1>0$, where both $\mathbb{B}(O,R_1)$ and
$\mathbb{B}(O,R_2)$ denote open balls with center $O$ and radii
$R_1$ and $R_2$ in $\mathbb{R}^n$, and $\Gamma_1$ and $\Gamma_2$
denote the corresponding boundary sphere, respectively. We are
interested in the radially symmetric and continuous solutions for
(3). Correspondingly, let the feasible function space $\mathscr{N}$
in (4) be defined as \beq\mathscr{N}:=\Big\{u\in
W_0^{1,p}(\Omega)\cap C(\overline{\Omega}):\ u\ \text{\rm is
radially symmetric}\Big\}.\eneq

In the following, one is to construct the radially symmetric
analytic solution for (3) through {\it canonical duality method}
introduced by G. Strang et al. This theory was originally proposed
to find minimizers for a non-convex strain energy functional with a
double-well potential. During the past few years, considerable
effort has been taken to illustrate these non-convex problems from
the theoretical point of view. By applying this technique, the
authors characterized the local and global energy extrema for both
hard and soft devices and finally
obtained the analytical solutions \cite{G6}.\\

At the moment, we would like to introduce the main theorem.
\begin{thm}
Assume that the source term $f\in C(\overline{\Omega})$ is radially
symmetric and satisfies \beq{\rm Sgn}(f)=1\ \text{\rm or}\ {\rm
Sgn}(f)=-1\ \text{\rm on}\ \overline{\Omega}.\eneq Moreover, let
\beq E_{\epsilon,p}(r):=\displaystyle
r^{(2p-2)/(p-2)}-\chi(p)\epsilon^2r^2. \eneq It is evident that,
when $p\in(1,2)$, $E_{\epsilon,p}$ is strictly decreasing with
respect to $r\in(0,\epsilon^{p-2}]$; while when $p\in(2,\infty)$,
$E_{\epsilon,p}$ is strictly increasing with respect to
$r\in[0,\infty)$. In either case, let $E_{\epsilon,p}^{-1}$ stand
for the inverse of $E_{\epsilon,p}$. Then, there exists a sequence
of global minimizers $\{\bar{u}_\epsilon\}_\epsilon$ from
$\mathscr{N}$ in (9) for the approximation problems
($\mathscr{P}^{(\epsilon)}$), which is at the same time a sequence
of analytic solutions for the Euler-Lagrange equations (3) in the
following explicit form $\bar{u}_\epsilon(r)$(without any confusion
with respect to $\bar{u}_\epsilon(x)$, as well as $f(r)$), \beq
\bar{u}_\epsilon(r)=\displaystyle\int^{r}_{R_1}F_\epsilon(\rho)\rho/E_{\epsilon,p}^{-1}(F_\epsilon^2(\rho)\rho^2)d\rho,\
\ r\in[R_1,R_2], \eneq where $F_\epsilon$ is defined as

\beq F_\epsilon(r):=\displaystyle
R_1^nC_\epsilon/r^n+\int_{r}^{R_1}f(\rho)\rho^{n-1}/r^n d\rho,\ \
r\in[R_1,R_2],\eneq and $\{C_\epsilon\}_\epsilon$ is a positive
number sequence given by (33).
\end{thm}

%\begin{thm}
%For any radially symmetric and positive density functions $f^+\in
%C(\overline{\Omega})$ and $f^-\in C(\overline{\Omega^*})$ satisfying
%the normalized balance condition, from a subsequence of
%$\{\bar{u}_k\}_k$, one can derive a global maximizer subject to
%(2)-(5) for the Monge-Kantorovich problem ($\mathcal{P}$).
%\end{thm}
\begin{rem}
For $p>2$, the above theorem gives the exact radially symmetric
solution of the $p$-Laplacian (1). While for $p<2$, using the
techniques in the proof of Theorem 1.1, we can choose from
$\{\bar{u}_\epsilon\}_\epsilon$ a subsequence of radially symmetric
solutions of (3) which converge to the exact radially symmetric
solution of (1) in $W^{1,p}_0(\Omega)$ as $\epsilon\to0$.
\end{rem}
\begin{rem}
From the structure of the number sequence $\{C_\epsilon\}_\epsilon$
in the proof of Theorem 1.2, one finds out, if ${\rm Sgn}(f)=1$,
then $\bar{u}_\epsilon\geq0$ on $\overline{\Omega}$, while if ${\rm
Sgn}(f)=-1$, then $\bar{u}_\epsilon\leq0$ on $\overline{\Omega}$,
which is in accordance with Theorem 1.1. In addition, the special
form of $\{C_\epsilon\}_\epsilon$ does not allow $R_1=0$.
\end{rem}
\begin{rem}
When $p=\infty$, it turns out $E_{\epsilon,\infty}=r^2$, which is
invertible on $[0,\infty)$. Then, one has $|du_\epsilon/dr|=1$, and
an exact solution is in the zigzag form. The uniqueness of the
solution fails in this limit case \cite{Evans1}. On the other hand,
for the case $p=1$, (8) in Theorem 1.1 fails. These two important
cases remain to be discussed \cite{Evans1,BK}.

\end{rem}
The rest of the paper is organized as follows. In Section 2, first,
we introduce some useful notations which will simplify our proofs
considerably. Then, we prove Theorem 1.1 by applying
Rellich-Kondrachov compactness theorem and Riesz's mean convergence
theorem. Next, we apply the canonical dual transformation to deduce
a sequence of perfect dual problems ($\mathscr{P}^{(\epsilon)}_d$)
corresponding to $(\mathscr{P}^{(\epsilon)})$ and a pure
complementary energy principle in order to prove Theorem 1.2.
\section{Proof of the main results}
\subsection{Notations}
Before proceeding, first and foremost, we introduce some useful
notations.\\

\begin{itemize}
%\item Define the indicator function $\chi$ as \[\chi(p):=\left\{\begin{array}{ll}
%0,&p\in(2,+\infty);\\
%\\
%1,&p\in(1,2).
%\end{array}\right.\]
\item $\overrightarrow{\theta_\epsilon}$ is the corresponding G\^{a}teaux derivative of
$H$ with respect to $\nabla u_\epsilon$ given by
\[\overrightarrow{\theta_\epsilon}(x)=(\theta^{(1)}_{\epsilon}(x),\cdots,\theta^{(n)}_{\epsilon}(x)):=(|\nabla
u_\epsilon|^2+\epsilon^2\chi(p))^{(p-2)/2}\nabla u_\epsilon.\]
\item $\Phi_\epsilon:\mathscr{N}\to L^{{p}/{2}}(\Omega)$ is a nonlinear
geometric mapping given by \[\Phi_\epsilon(u_\epsilon):=|\nabla
u_\epsilon|^{2}+\epsilon^2\chi(p).\] For convenience's sake, denote
$\xi_\epsilon:=\Phi_\epsilon(u_\epsilon).$
\item $\Psi$ is the canonical energy
defined as
\[\Psi(\xi_\epsilon):={\xi_\epsilon}^{{p}/{2}}/p.\]
\item $\zeta_\epsilon$ is the corresponding G\^{a}teaux derivative of
$\Psi$ with respect to $\xi_\epsilon$ given by
\[\zeta_\epsilon=\xi_\epsilon^{{(p-2)}/{2}}/{2},\] which is
invertible with respect to $\xi_\epsilon$ and belongs to the
following function space $\mathscr{V}_\epsilon$,
$$\mathscr{V}_\epsilon:=\left\{\begin{array}{cl}
\Big\{\phi: 0<\phi\leq\epsilon^{p-2}/2\Big\},&
p\in(1,2);\\
\\
\Big\{\phi: \phi\in L^{p/(p-2)}(\Omega)\Big\},&
p\in(2,\infty).\end{array}\right.$$
\item $\Psi_\ast$ is Legendre transformation defined
as
\[\Psi_\ast(\zeta_\epsilon):=\xi_\epsilon\zeta_\epsilon-\Psi(\xi_\epsilon)=(p-2)(2\zeta_\epsilon)^{p/(p-2)}/(2p).\]
\end{itemize}
\subsection{Proof of Theorem 1.1}
Without loss of generality, we prove the case for Sgn$(f)=1$. Let
$\bar{u}_\epsilon$ be a solution for (4), then, by the minimization
property, $\bar{u}_\epsilon\geq0$ a.e. in $\Omega$. Next, we prove
the existence and convergence of a sequence of solutions. By
applying H\"{o}lder's inequality, one has
\[\displaystyle\int_\Omega |fu_\epsilon|
dx\leq\Big(\int_\Omega|f|^{p^*}dx\Big)^{1/p^*}\Big(\int_\Omega|u_\epsilon|^{p}dx\Big)^{1/p}.\]
%By applying Young's inequalty, one has
%\[\displaystyle|\int_\Omega fu_\epsilon
%dx|\leq1/p\int_\Omega|u_\epsilon|^pdx+1/p^*\int_\Omega
%|f|^{p^*}dx.\]
Since $\Omega$ is bounded, for any $u_\epsilon\geq0$ a.e. from
$W^{1,p}_0(\Omega)$, by applying Poincar\'{e}'s inequality, one
obtains
\[\begin{array}{lll}\epsilon^p/p\ {\rm meas}(\Omega)&\geq&\displaystyle\int_\Omega H_\epsilon(\nabla
u_\epsilon)dx-\int_\Omega fu_\epsilon dx\\
\\
&\geq&\displaystyle 1/p\int_\Omega (|\nabla
u_\epsilon|^2+\epsilon^2)^{p/2}dx-\Big(\int_\Omega|f|^{p^*}dx\Big)^{1/p^*}\Big(\int_\Omega|u_\epsilon|^{p}dx\Big)^{1/p}\\
\\
&\geq&\displaystyle 1/p\int_\Omega (|\nabla
u_\epsilon|^2+\epsilon^2)^{p/2}dx-C(p)\Big(\int_\Omega|f|^{p^*}dx\Big)^{1/p^*}\Big(\int_\Omega|\nabla
u_\epsilon|^{p}dx\Big)^{1/p}.
%&\geq&\displaystyle 1/p\int_\Omega (|\nabla
%u_\epsilon|^2+\epsilon^2)^{p/2}dx-1/p\int_\Omega|u_\epsilon|^pdx-1/p^*\int_\Omega
%|f|^{p^*}dx\\
%\\
%&\geq&\displaystyle 1/p\int_\Omega (|\nabla
%u_\epsilon|^2+\epsilon^2)^{p/2}dx-C(p)\int_\Omega|\nabla
%u_\epsilon|^pdx-1/p^*\int_\Omega
%|f|^{p^*}dx\\
\\
\end{array}\]
Using the contradiction method, it is easy to check that,
$u_\epsilon$ is uniformly bounded in $W^{1,p}_0(\Omega)$ for $p>1$,
and consequently, $I^{(\epsilon)}$ is bounded from below.\\

We see that, there exists at least one solution $\bar{u}_\epsilon\in
W^{1,p}_0(\Omega)$ to (4) due to the fact that
$L_{\epsilon}(\overrightarrow{\gamma},z,x)$ satisfies the coercivity
inequality (5) and is convex in the variable
$\overrightarrow{\gamma}$ \cite{Evans}. On the one hand, according
to the lower semicontinuity of $I^{(\epsilon)}$ and the unique
solvability of (2), it holds \beq \bar{u}_\epsilon\rightharpoonup
\bar{u}\ \text{weakly\ in}\ W^{1,p}_0(\Omega)\ \text{as}\
\epsilon\to0, \eneq where $\bar{u}$ is the unique solution to (2).
Moreover, since the imbedding $W^{1,p}(\Omega)\hookrightarrow
L^1(\Omega)$ is compact, then, \beq \bar{u}_\epsilon\rightarrow
\bar{u}\ \text{in}\ L^{1}(\Omega)\ \text{as}\ \epsilon\to0.\eneq On
the other hand, from the minimization property $I[\bar{u}]\leq
I[\bar{u}_\epsilon]\leq I^{(\epsilon)}[\bar{u}_\epsilon]\leq
I^{(\epsilon)}[\bar{u}]$, we have \beq I[\bar{u}_\epsilon]\to
I[\bar{u}]\ \text{as}\ \epsilon\to0. \eneq The above convergence
properties (15) and (16) indicate that \beq
\displaystyle\int_\Omega|\nabla \bar{u}_\epsilon|^pdx\to
\int_\Omega|\nabla \bar{u}|^pdx\ \text{as}\ \epsilon\to0.\eneq
Finally, combining (14) and (17), we reach the conclusion.
\subsection{Proof of Theorem 1.2}
\begin{defi}
By Legendre transformation, one defines a total complementary energy
functional $\Xi$,
\[
\Xi(u_\epsilon,\zeta_\epsilon):=\displaystyle\int_{\Omega}\Big\{\Phi_{\epsilon}(u_\epsilon)\zeta_\epsilon-\Psi_\ast(\zeta_\epsilon)
-fu_\epsilon\Big\}dx.
\]
\end{defi}
Next, we introduce an important {\it criticality criterium} for the
total complementary energy functional.
\begin{defi}
$(\bar{u}_\epsilon, \bar{\zeta}_\epsilon)$ is called a critical pair
of $\Xi$ if and only if \beq
D_{u_\epsilon}\Xi(\bar{u}_\epsilon,\bar{\zeta}_\epsilon)=0, \eneq
\beq D_{\zeta_\epsilon}\Xi(\bar{u}_\epsilon,\bar{\zeta}_\epsilon)=0,
\eneq where $D_{u_\epsilon}, D_{\zeta_\epsilon}$ denote the partial
G\^ateaux derivatives of $\Xi$, respectively. \end{defi} Indeed, by
variational calculus, one has the following observations from (18)
and (19).
\begin{lem}
On the one hand, for any fixed
$\zeta_\epsilon\in\mathscr{V}_\epsilon$, $(18)$ is equivalent to the
equilibrium equation
\[
\begin{array}{ll}\displaystyle {\rm div}(2\zeta_\epsilon \nabla\bar{u}_\epsilon)+f=0,& \
\text{\rm in}\ \Omega.\end{array}
\]
On the other hand, for any fixed $u_\epsilon$ from $\mathscr{N}$,
(19) is consistent with the constructive law
\[
\Phi_\epsilon(u_\epsilon)=D_{\zeta_\epsilon}\Psi_\ast(\bar{\zeta}_\epsilon).
\]
\end{lem}
Lemma 2.3 indicates that $\bar{u}_\epsilon$ from the critical pair
$(\bar{u}_\epsilon,\bar{\zeta}_\epsilon)$ solves the Euler-Lagrange
equation (3).
\begin{defi}
From Definition 2.2, one defines the pure complementary energy
$I^{(\epsilon)}_d$ in the following form
\[
I^{(\epsilon)}_d[\zeta_\epsilon]:=\Xi(\bar{u}_\epsilon,\zeta_\epsilon),
\]
where $\bar{u}_\epsilon$ solves the Euler-Lagrange equation (3).
\end{defi}
To further the discussion, one uses another representation of the
pure energy $I^{(\epsilon)}_d$ given by the following lemma.
\begin{lem} The
pure complementary energy functional $I^{(\epsilon)}_d$ can be
rewritten as
\[
I^{(\epsilon)}_d[\zeta_\epsilon]=-1/2\int_{\Omega}\Big\{{|\overrightarrow{\theta_\epsilon}|^2/(2\zeta_\epsilon)}-2\chi(p)\epsilon^2\zeta_\epsilon+(p-2)(2\zeta_\epsilon)^{p/(p-2)}/p\Big\}dx,
\]
where $\overrightarrow{\theta_\epsilon}$ satisfies \beq {\rm
div}\overrightarrow{\theta_{\epsilon}}+f=0\ \ \text{\rm in}\ \Omega,
\eneq equipped with a hidden boundary condition.
\end{lem}
\begin{proof}
Through integrating by parts, one has
\[
\begin{array}{lll}
I^{(\epsilon)}_d[\zeta_\epsilon]&=&\displaystyle-\underbrace{\int_\Omega\Big\{{\rm
div}(2\zeta_\epsilon \nabla\bar{u}_{\epsilon})+f\Big\}\bar{u}_\epsilon dx}_{(I)}\\
\\
&&-\underbrace{\int_\Omega\Big\{\zeta_\epsilon|\nabla\bar{u}_{\epsilon}|^2-\chi(p)\epsilon^2\zeta_\epsilon+(p-2)(2\zeta_\epsilon)^{p/(p-2)}/(2p)\Big\}dx.}_{(II)}\\
\\
\end{array}
\]
Since $\bar{u}_\epsilon$ solves the Euler-Lagrange equation (3),
then the first part $(I)$ disappears. Keeping in mind the
definitions of $\overrightarrow{\theta}_\epsilon$ and
$\zeta_\epsilon$, one reaches the conclusion.
\end{proof}

With the above discussion, in the following, we establish a sequence
of dual variational problems ($\mathscr{P}_d^{(\epsilon)}$)
corresponding to the approximation problems
($\mathscr{P}^{(\epsilon)}$), namely,
\begin{itemize}
\item $p\in(1,2)$
\beq
\displaystyle\min_{\zeta_\epsilon\in\mathscr{V}_{\epsilon}}\Big\{I^{(\epsilon)}_d[\zeta_\epsilon]=-1/2\int_{\Omega}\{{|\overrightarrow{\theta_\epsilon}|^2/(2\zeta_\epsilon)}-2\chi(p)\epsilon^2\zeta_\epsilon+(p-2)(2\zeta_\epsilon)^{p/(p-2)}/p\}dx\Big\};
\eneq
\item $p\in(2,\infty)$
\beq
\displaystyle\max_{\zeta_\epsilon\in\mathscr{V}_{\epsilon}}\Big\{I^{(\epsilon)}_d[\zeta_\epsilon]=-1/2\int_{\Omega}\{{|\overrightarrow{\theta_\epsilon}|^2/(2\zeta_\epsilon)}-2\chi(p)\epsilon^2\zeta_\epsilon+(p-2)(2\zeta_\epsilon)^{p/(p-2)}/p\}dx\Big\}.
\eneq
\end{itemize}
Indeed, by calculating the G\^{a}teaux derivative of
$I_d^{(\epsilon)}$ with respect to $\zeta_\epsilon$, one has
\begin{lem} The variation of $I_d^{(\epsilon)}$ with respect to $\zeta_\epsilon$
leads to the Dual Algebraic Equation(DAE), namely, \beq
|\overrightarrow{\theta_\epsilon}|^2=(2\bar{\zeta}_\epsilon)^{(2p-2)/(p-2)}-\chi(p)\epsilon^2(2\bar{\zeta}_\epsilon)^2,
\eneq where $\bar{\zeta}_\epsilon$ is from the critical pair
$(\bar{u}_\epsilon,\bar{\zeta}_\epsilon)$.
\end{lem}
Let $\lambda_\epsilon:=2\bar{\zeta}_\epsilon$,  the identity (23)
can be rewritten as \beq
|\overrightarrow{\theta_\epsilon}|^2=E_{\epsilon,p}(\lambda_\epsilon)=\lambda_\epsilon^{(2p-2)/(p-2)}-\chi(p)\epsilon^2\lambda_\epsilon^2.
\eneq From the above discussion, one deduces that, once
$\overrightarrow{\theta_\epsilon}$ is given, then the analytic
solution of the Euler-Lagrange equation (3) can be represented as
\beq
\bar{u}_k(x)=\displaystyle\int^{x}_{x_0}\overrightarrow{\eta_\epsilon}\overrightarrow{dt},
\eneq where $x\in \overline{\Omega}, x_0\in\Gamma$,
$\overrightarrow{\eta_\epsilon}=(\eta^{(1)}_\epsilon,\eta^{(2)}_\epsilon,\cdots,\eta^{(n)}_\epsilon):=\overrightarrow{\theta_\epsilon}/\lambda_\epsilon$,
which satisfies the condition for path independent integrals,
namely,
\[
\partial_{x_i}\eta_\epsilon^{(j)}-\partial_{x_j}\eta_{\epsilon}^{(i)}=0,\
\ i,j=1,\cdots,n.
\]
At the moment, we verify that $\bar{u}_\epsilon$ is exactly a global
minimizer over $\mathscr{N}$ for ($\mathscr{P}^{(\epsilon)}$) and
$\bar{\zeta}_\epsilon$ is a global extremum over
$\mathscr{V}_\epsilon$ for ($\mathscr{P}^{(\epsilon)}_d$).
\begin{lem}(Canonical duality theory)
Assume that $f\in C(\overline{\Omega})$ is radially symmetric and
satisfies ${\rm Sgn}(f)=1$ or ${\rm Sgn}(f)=-1$, then, there exists
a unique sequence of radially symmetric solutions
$\{\bar{u}_\epsilon\}_\epsilon$ for the Euler-Lagrange equations (3)
with Dirichlet boundary in the form of (25), which is at the same
time a sequence of global minimizers over $\mathscr{N}$ for the
approximation problems ($\mathscr{P}^{(\epsilon)}$). And the
corresponding $\{\bar{\zeta}_\epsilon\}_\epsilon$ is a sequence of
global minimizers for the dual problems (21), or a sequence of
global maximizers for the dual problems (22). Moreover, the
following duality identities hold, \beq
I^{(\epsilon)}[\bar{u}_\epsilon]=\displaystyle\min_{u_\epsilon\in\mathscr{N}}I^{(\epsilon)}[u_\epsilon]=\Xi(\bar{u}_\epsilon,\bar{\zeta}_\epsilon)=\displaystyle\min_{\zeta_\epsilon\in\mathscr{V}_{\epsilon}}I_d^{(\epsilon)}[\zeta_\epsilon]=I_d^{(\epsilon)}[\bar{\zeta}_\epsilon]\
\text{for}\ p\in(1,2); \eneq
\item
\beq
I^{(\epsilon)}[\bar{u}_\epsilon]=\displaystyle\min_{u_\epsilon\in\mathscr{N}}I^{(\epsilon)}[u_\epsilon]=\Xi(\bar{u}_\epsilon,\bar{\zeta}_\epsilon)=\displaystyle\max_{\zeta_\epsilon\in\mathscr{V}_{\epsilon}}I_d^{(\epsilon)}[\zeta_\epsilon]=I_d^{(\epsilon)}[\bar{\zeta}_\epsilon]\
\text{for}\ p\in(2,\infty). \eneq
\end{lem}
\begin{rem}
Lemma 2.7 shows that the minimization(maximization) of the pure
complementary energy functional $I_d^{(\epsilon)}$ is perfectly dual
to the minimization of the potential energy functional
$I^{(\epsilon)}$ for $p\in(1,2)$($p\in(2,\infty)$). Indeed, both
identities (26) and (27) indicate there is no duality gap between
them.
\end{rem}
\begin{proof}
The proof is divided into two parts. In the first parts, we discuss
the uniqueness of $\overrightarrow{\theta_\epsilon}$ for both ${\rm
Sgn}(f)=1$ and ${\rm Sgn}(f)=-1$. Global extrema will be
studied in the second part.\\
\\
{\it First Part:}\\

Let $O=(a_1,a_2,\cdots,a_n)$. Actually, due to the radial symmetry
of the solution $\bar{u}_\epsilon$ of (3), one has, \beq
\overrightarrow{\theta_\epsilon}={F}_\epsilon(r)(x_1-a_1,\cdots,x_n-a_n)={F}_\epsilon\Big(\sqrt{\sum_{i=1}^n(x_i-a_i)^2}\Big)(x_1-a_1,\cdots,x_n-a_n),
\eneq where ${F}_\epsilon(r)$ in (13)
%\[
%\overline{F}_\epsilon(r)=C_\epsilon
%R_1^n/r^n+\int_r^{R_1}f(\rho)\rho^{n-1}/r^n d\rho
%\]
is a general solution of the nonhomogeneous linear differential
equation \beq {F}_\epsilon'(r)+n{F}_\epsilon(r)/r=-f(r)/r,\ \ \
r\in(R_1,R_2). \eneq Recall that $\bar{u}_\epsilon(R_1)=0$, as a
result, \beq
\bar{u}_\epsilon(r)=\int_{R_1}^r\Big(R_1^nC_\epsilon+\int_{\rho}^{R_1}f(r)r^{n-1}
dr\Big)/\Big(\rho^{n-1}\lambda_\epsilon(\rho)\Big)d\rho, \ \ \ \
r\in(R_1,R_2). \eneq Recall that
\[
\bar{u}_\epsilon(R_2)=0,
\]
and one can determine the positive constant $C_\epsilon$ uniquely.
Indeed, let \beq \mu_\epsilon(r,y):=\Big({\rm Sgn}(f)y
R_1^n+\int_r^{R_1}f(\rho)\rho^{n-1}
d\rho\Big)/\Big(r^{n-1}\lambda_\epsilon(r,y)\Big), \eneq
%\[
%\eta_\epsilon(r,y):=\Big(-y R_1^n+\int_r^{R_1}f(\rho)\rho^{n-1}
%d\rho\Big)/\Big(r^{n-1}\lambda_\epsilon(r,y)\Big),
%\]
\beq M_{\epsilon}(y):=\int^{R_2}_{R_1}\mu_\epsilon(\rho,y)d\rho,
\eneq
%\[
%N_{\epsilon}(y):=\int^{R_2}_{R_1}\eta_\epsilon(\rho,y)d\rho,
%\]
where $\lambda_\epsilon(r,y)$ is from (24). It is evident that
$\lambda_\epsilon$ depends on $C_\epsilon$. As a matter of fact,
similar as the discussion in \cite{LW}, one can check that, if ${\rm
Sgn}(f)=1$, then $M_\epsilon$ is strictly increasing with respect to
$y\in(0,\infty)$; while if ${\rm Sgn}(f)=-1$, then $M_\epsilon$ is
strictly decreasing with respect to $y\in(0,\infty)$. As a result,
for both cases, one has \beq C_\epsilon=M_{\epsilon}^{-1}(0). \eneq
The uniqueness of the solution $\bar{u}_\epsilon$ is concluded.\\
\\
{\it Second Part:}\\

On the one hand, for any test function $\phi\in
W_0^{1,\infty}(\Omega)$, the second variational form
$\delta_\phi^2I^{(\epsilon)}$ is equal to\beq
\int_\Omega\Big\{(p-2)\Big(\Big(|\nabla
\bar{u}_\epsilon|^2+\chi(p)\epsilon^2\Big)^{(p-4)/2}|\nabla\bar{u}_\epsilon\cdot\nabla\phi|^2\Big)+\Big(|\nabla
\bar{u}_\epsilon|^2+\chi(p)\epsilon^2\Big)^{(p-2)/2}|\nabla\phi|^2\Big\}dx.\eneq
On the other hand, for any test function
$\psi\in\mathscr{V}_{\epsilon}$, the second variational form
$\delta_\psi^2I_d^{(\epsilon)}$ is equal to
\beq-1/2\int_\Omega\Big\{|\overrightarrow{\theta_\epsilon}|^2/\bar{\zeta}_\epsilon^3+8\cdot(2\bar{\zeta}_\epsilon)^{(4-p)/(p-2)}/(p-2)\Big\}\psi^2dx.
\eneq From (34), by keeping in mind the fact
\[|\nabla\bar{u}_\epsilon\cdot\nabla\phi|^2\leq|\nabla\bar{u}_\epsilon|^2|\nabla\phi|^2,\] one deduces
immediately that, for $p\in(1,2)\cup(2,+\infty)$, \beq \delta^2_\phi
I^{(\epsilon)}(\bar{u}_\epsilon)\geq0. \eneq From (35), taking the
definition of $\overrightarrow{\theta_\epsilon}$ into account, one
has \beq\delta_\psi^2 I_d^{(\epsilon)}(\bar{\zeta}_\epsilon)\geq0,\
\  p\in(1,2)\eneq and \beq \delta_\psi^2
I_d^{(\epsilon)}(\bar{\zeta}_\epsilon)\leq0,\ \ p\in(2,+\infty).
\eneq
\end{proof}
Summarizing the above discussion, one proves Theorem 1.2.\\
\\
{\bf Acknowledgment}: This project is partially supported by US Air
Force Office of Scientific Research (AFOSR FA9550-10-1-0487),
Natural Science Foundation of Jiangsu Province (BK 20130598),
National Natural Science Foundation of China (NSFC 71273048,
71473036, 11471072), the Scientific Research Foundation for the
Returned Overseas Chinese Scholars, Fundamental Research Funds for
the Central Universities on the Field Research of Commercialization
of Marriage between China and Vietnam (No. 2014B15214). This work is
also supported by Open Research Fund Program of Jiangsu Key
Laboratory of Engineering Mechanics, Southeast University
(LEM16B06).

\end{document}